\documentclass[10pt]{amsart}

\usepackage{amsmath}
\usepackage{amssymb}
\usepackage{amsrefs}
\usepackage{verbatim}
\usepackage{enumerate}
\usepackage[all]{xy}
\usepackage{mathdots}

\numberwithin{equation}{section}
\newtheorem{thm}[equation]{Theorem}
\newtheorem*{thm*}{Theorem}

\newtheorem{lemma}[equation]{Lemma}
\newtheorem{ex}[equation]{Example}

\theoremstyle{remark}
\newtheorem*{remark*}{Remark}
\newtheorem{remark}[equation]{Remark}

\theoremstyle{definition}

\newcommand{\exref}[1]{Ex\-am\-ple \ref{#1}}
\newcommand{\thmref}[1]{Theo\-rem \ref{#1}}

\newcommand{\lemref}[1]{Lem\-ma \ref{#1}}

\newcommand{\remref}[1]{Re\-mark \ref{#1}}

\providecommand{\normaleq}{\unlhd}

\DeclareMathOperator{\Hol}{Hol}

\DeclareMathOperator{\End}{End}

\DeclareMathOperator{\Adj}{Adj }
\DeclareMathOperator{\Der}{Der }

\DeclareMathOperator{\GL}{GL}

\DeclareMathOperator{\gl}{\mathfrak{gl}}

\DeclareMathOperator{\Aut}{Aut}

\DeclareMathOperator{\Cent}{Cen}

\title{Longer nilpotent series}

\author{James B. Wilson}
\address{
	Department of Mathematics\\
	Colorado State University\\
	Fort Collins, CO 80523\\
}
\email{jwilson@math.colostate.edu}
\date{\today}

\begin{document}

\maketitle

\begin{abstract}
New nilpotent series are produced that refine the usual nilpotent series of a
group.  These refinements can be arbitrarily longer than the
series they refine and therefore clarify in greater detail the structure of
automorphisms of nilpotent groups. Examples, properties, and an application 
to group-isomorphism testing are provided.
\end{abstract}

\section{Introduction}

In order to describe the automorphisms of a finite group
we often begin by locating a characteristic series from which 
the automorphisms can be determined by a recursive process.
This thinking is also used in group-isomorphism tests and classification,
as seen in early work by Fitting and Hall \citelist{\cite{Fitting:const}\cite{Hall:const}*{p. 208}},  and in successive 
improvements, e.g. \citelist{\cite{Higman:chic}*{Section III--IV}\cite{Newman}\cite{Robinson:aut}\cite{ELGO}*{Section 7}\cite{CH:iso}\cite{Babai:iso}}.

A barrier is that many groups have  few known characteristic subgroups.  
Obviously products of isomorphic simple groups have no proper nontrivial 
characteristic subgroups. Taunt and Glasby-P{\'a}lfy-Schneider characterized 
groups with a unique proper nontrivial characteristic 
subgroup \citelist{\cite{Taunt}\cite{GPS}}.  Yet those situations seem rare when
compared to the complexity of general finite groups.

Evidence suggests that  $p$-groups have many 
characteristic subgroups beyond those typically known.  Martin and Helleloid \citelist{\cite{Martin}\cite{HM:auto-generic}} 
show that for `most' finite $p$-groups $G$, $\Aut(G)$ is also a $p$-group.\footnote{`Most'
in those works is the conditional probability after fixing natural properties of $G$
\cite{HM:auto-generic}.}
So the action of $\Aut(G)$ on the factors of the exponent-$p$ central series of $G$ stabilizes
a maximal flag of each factor.  Remarkably, $G$ has a characteristic composition series (the preimages of the flags).

This article introduces characteristic refinements of nilpotent
series that can be located by solving systems of linear equations.
These refinements retain a correspondence with Lie rings graded by commutative monoids.
Such monoids capture  complicated subgroup containment.
Repeating the methods creates characteristic series substantially longer than traditional 
verbal and marginal subgroups chains.

\section{Notation}

Here $\mathbb{N}$ is the non-negative integers and $\mathbb{Z}^+=\mathbb{N}-\{0\}$.
For a set ${\tt X}$, $2^{\tt X}$ is its power set.  Our use of groups and rings follows \citelist{\cite{Khukhro}\cite{Jac:basicII}*{Chapter 4}}.
For $x,y\in G$, $[x,y]=x^{-1} x^y=x^{-1}y^{-1}xy$ and for $X,Y\subseteq G$, $[X,Y]=\langle [x,y] :x\in X, y\in Y\rangle$.
For subgroups $A_i$, $[A_1]=A_1$ and $[A_1,\dots, A_{m+1}]=[[A_1,\dots,A_m], A_{m+1}]$. 
A left $\mathbb{Z}$-module $V$ is a right $\End(V)$-module and a left $\End(V)^{\rm op}$-module
(where $\End(V)$ is the endomorphism ring of $V$ and $\End(V)^{\rm op}$ its opposite ring).
Put $\mathbb{Z}_p=\mathbb{Z}/p\mathbb{Z}$ and $\mathfrak{gl}(V)=\End(V)$ with
product $[X,Y]=XY-YX$.

\section{Filters}

Fix a group $G$ and a commutative monoid $M$.  A {\em filter} on $G$ is a function $\phi:M\to 2^G$  where
for every $s\in M$, $\phi_s$ is a subgroup of $G$, $G=\phi_0$, and
\begin{align}\label{eq:def-filter}
	(\forall & s,t\in M) & [\phi_s,\phi_t] & \leq \phi_{s+t}\leq \phi_s\cap \phi_t.
\end{align}
The assumption $G=\phi_0$ is not necessary but convenient since all observations about filters occur
between the groups $\cap_{s\in M} \phi_s$ and $\phi_0$.  Note that $\phi_s\normaleq \phi_0$.
Associated to a filter $\phi:M\to 2^G$ are the following normal subgroups: for $s\in M$, 
\begin{align*}
	 \phi_{s}^+ & = \prod_{t\in M-\{0\}} \phi_{s+t}.
\end{align*}
For all $s\in M$, $\phi_s^+\leq \phi_s$ and if $M=\langle {\tt X}\rangle$ then
$\phi_s^+=\langle \phi_{s+x}: x\in {\tt X}-\{0\}\rangle$.  
Notice 
\begin{align*}
	(\forall & s,t\in M) & [\phi_s^+,\phi_t]
		 & = \prod_{u\in M-\{0\}}[\phi_{s+u},\phi_{t}] 
		 \leq \prod_{u\in M-\{0\}} \phi_{s+u+t}
		  = \phi_{s+t}^+.
\end{align*}
Likewise, $[\phi_s,\phi_t^+]\leq \phi_{s+t}^+$.  
Now, if $s\in M-\{0\}$ then 
$[\phi_s,\phi_s]\leq \phi_{s+s}\leq \phi_s^+$ and so $L_s=\phi_s/\phi_{s}^+$ is abelian.
As $[\phi_0^+,\phi_s]\leq \phi_s^+$, $L_s$ 
is a right $\mathbb{Z}[\phi_0/\phi_0^+]$-module where $\phi_0/\phi_0^+$ acts by conjugation.
Associate to $\phi$ the abelian group:
\begin{align}\label{eq:Lie-filter}
 	L(\phi) & = \bigoplus_{s\in M} L_s, & L_0=0.
\end{align}
Also define an $M$-graded product on the homogeneous components by
\begin{align*}
	(\forall & x\in \phi_s,\forall y\in \phi_t) &
		[x\phi_s^+,y\phi_t^+]_{st} & = [x,y]\phi_{s+t}^+ = x^{-1}x^y\phi_{s+t}.
\end{align*}
\begin{thm}\label{thm:Lie}
$L(\phi)$ is an $M$-graded Lie ring and a $\mathbb{Z}[\phi_0/\phi_0^+]$-module.
\end{thm}
\begin{proof} Compare \cite{Lazard}*{Chapter I}. 
\end{proof}
(\thmref{thm:Lie} holds letting $L_0$ be a Lie ring of  derivations on $\bigoplus_{s\in M-\{0\}} L_s$.)

Filters with $M\cong \mathbb{N}$ are essentially the filters described by Lazard \cite{Lazard}*{p. 106} but
with explicit operators.  For example, the  lower central series $\gamma$ of a group $N$ is
\begin{align}
	(\forall & i\in \mathbb{Z}^+) & \gamma_i & = \overbrace{[N,\dots,N]}^i.
\end{align}
To extend $\gamma$ to a filter on $\mathbb{N}$ we have several options, e.g. let $\gamma_0=N$.  A more
informative choice is to let $\gamma_0=\Hol(N)=\Aut(N)\ltimes N$ be the {\em holomorph} of $N$.  This captures the property
that for $i>0$, $\gamma_i$ is characteristic in $N$.  If $G$ is nilpotent of class $c$ and $|G|\neq 2$ then 
$\gamma$ factors through the injective filter on $\{0,1,\dots, c=c+1\}$.  

More generally, given a group $G$ and normal subgroup $N$,
the map $\gamma:\mathbb{N}\to 2^G$ with $\gamma_0=G$ and $\gamma_i=\gamma_i(N)$ for $i>0$ is a filter
where the subgroups $\gamma_1\geq \gamma_2\geq\cdots$ are a nilpotent series of $G$-invariant subgroups of $N$.
 This same treatment applies to Higman's exponent-$p$ central
series $\eta$, and the Jenning's series $\kappa$, i.e.: for $N\normaleq G$, set $G=\eta_0=\kappa_0$, $N=\eta_1=\kappa_1$ 
and recursively define for each $i>0$, 
\begin{align}
	 \eta_{i+1}(N) & = [N,\eta_i(N)]\eta_i^p(N)\qquad\& &
	 \kappa_i(N) = [N,\kappa_{i-1}(N)]\kappa_{\lfloor i/p\rfloor}(N)^p.
\end{align}
In these cases it makes sense to use $\mathbb{Z}_p\otimes L(\eta)$ and $\mathbb{Z}_p\otimes L(\kappa)$  to obtain
graded Lie $\mathbb{Z}_p$-algebras.  Indeed,  $\mathbb{Z}_p\otimes L(\kappa)$ is $p$-restricted.
See \citelist{\cite{Khukhro}*{Chapter 3}\cite{Shalev:p-groups}} for surveys of
filters over $\mathbb{N}$, their properties, and their uses.   

\subsection{Filters over ordered monoids}\label{sec:order}
In a commutative monoid $M$ there is a natural reflexive and transitive relation $\prec$ (a {\em pre-order}) defined as
$s\prec u$ if there is a $t$ where $s+t=u$.
Notice that filters $\phi:M\to 2^G$ are order-reversing maps from $\langle M,\prec\rangle$ to $\langle 2^G, \subseteq\rangle$.
Hence, filters translate some of the often complicated subgroup inclusions in a group into the language of commutative monoids.

Notice if $s\prec t$ and $t\prec s$ then $\phi_{s}=\phi_t$.  So we can improve our understanding when $M$ is an ordered monoid, that is, there is a partial order $\leq$ on $M$
such that whenever $s\leq t$ and $u\leq v$ then also $s+t\leq u+v$.
Say a filter $\phi:M\to 2^G$ is {\em ordered} if $M$ is ordered and $s\leq t$ implies $\phi_s \geq \phi_t$.
We will sometimes call filters {\em pre-ordered filters} for added clarity.
Of particular interest to us are ordered filters over totally ordered commutative monoids $M$.   In such a filter
for every $s,t\in M$, either  $\phi_s\geq \phi_t$ or $\phi_s\leq \phi_t$, i.e.
$\{\phi_s : s\in M\}$ is a series.  Indeed, for every $s\in M$, there is an $s^+\in M$ with $\phi_s^+=\phi_{s^+}$ ($s^+$ may
not be unique).    If $M$ is well-ordered, then we may take $s^+=s+e$, $e=\min M-\{0\}$.  We call
an ordered filter on a well-ordered set a {\em $\nu$ series}.  

\subsection{Generating filters}\label{sec:gen-filter}

It will be convenient to specify filters by describing a few members which ``generate'' the 
remaining terms.  At issue is what generation should mean.
The monoid is an obvious resource.  Given generators ${\tt X}$ of a commutative monoid $M$ 
it would seem that a function $\pi:{\tt X}\to 2^G$ would be enough information to specify a corresponding
filter $\bar{\pi}:M\to 2^G$.  The complication is that \eqref{eq:def-filter} asks for $\bar{\pi}$ to satisfy both a lower and
upper bound.  This is possible with some assumptions on $(M,{\tt X}, \pi)$ but we are not aware of a
general meaning of generating a filter from an arbitrary function $\pi:{\tt X}\to 2^G$.

Fix a monoid $M$ and a set ${\tt X}$ that generates $M$.  The Cayley graph $\mathcal{G}=\mathcal{G}(M;{\tt X})$ has
vertex set $M$ and directed labeled edge set $\left\{s\overset{x}{\longrightarrow} s+x : s\in M, x\in {\tt X}\right\}$.  A finite directed path in the
Cayley graph from $0$ to a vertex $s$ is specified by a sequence $s_1,\dots,s_d$ in ${\tt X}$ 
where $s=s_1+\cdots+s_d$.  We let $\mathcal{G}_0^s$ denote the set of all finite directed paths from $0$ to $s$.
Note that we regard an element $u$ of $\mathcal{G}_0^s$ to be the sequence of edge labels, i.e. $u=(s_1,\dots,s_d)$
where $s=s_1+\cdots +s_d$.  For notation we write $[\pi_u]=[\pi_{s_1},\dots,\pi_{s_d}]$.  

For a function $\pi:{\tt X}\to 2^G$, define $\bar{\pi}:M\to 2^G$ as follows: for each $s\in S$,
\begin{align}\label{eq:bar}
	\bar{\pi}_s & = \prod_{u\in \mathcal{G}_0^s} [\pi_{u}].
\end{align}
Notice by \eqref{eq:def-filter}, if ${\tt X}=M$ and $\pi$ is a filter then $\bar{\pi}=\pi$.
By applying the $3$-subgroups lemma we show $\bar{\pi}$ already satisfies the first
inequality in \eqref{eq:def-filter}.

\begin{lemma}\label{lem:lower}
If $\pi:{\tt X}\to 2^G$ maps into the normal subgroups of $G$ then 
for every $s,t\in M$, $[\bar{\pi}_s, \bar{\pi}_t] \leq \bar{\pi}_{s+t}$. 
\end{lemma}
\begin{proof}
We begin by proving that for all $i,j\in\mathbb{Z}^+$ and all $s_1,\dots,s_{i+j}\in M$,
\begin{align}\label{eq:collect}
\left[[\bar{\pi}_{s_1}, \dots,\bar{\pi}_{s_i}],
	[\bar{\pi}_{s_{i+1}},\dots,\bar{\pi}_{s_{i+j}}]\right]
& \leq \prod_{\sigma \in S_{i+j}}
	 [\bar{\pi}_{s_{1\sigma}},\dots, \bar{\pi}_{s_{(i+j)\sigma}}].
\end{align}
We induct on $(i,j)$ where $\mathbb{Z}^{+}\times\mathbb{Z}^+$ is well-ordered by $(i,j)\leq (i',j')$ if $j<j'$, or $j=j'$ and $i\leq i'$.
For every $i\geq 1$, if $j=1$ then:  
\begin{align*}
\left[[\bar{\pi}_{s_1},\dots,\bar{\pi}_{s_i}],
      [\bar{\pi}_{s_{i+1}}]\right]
& = [\bar{\pi}_{s_1},\dots,\bar{\pi}_{s_i},\bar{\pi}_{s_{i+1}}]
& \leq \prod_{\sigma \in S_{i+1}}
	 [\bar{\pi}_{s_{1\sigma}},\dots, \bar{\pi}_{s_{(i+1)\sigma}}].
\end{align*}
Now suppose $j>1$.  Let $X=[\bar{\pi}_{s_1}, \dots,\bar{\pi}_{s_i}]$, $Y=[\bar{\pi}_{s_{i+1}},\dots,\bar{\pi}_{s_{i+j-1}}]$, 
and $Z=\bar{\pi}_{s_{i+j}}$.  It follows that:
\begin{align*}
\left[[\bar{\pi}_{s_1}, \dots,\bar{\pi}_{s_i}],
	[\bar{\pi}_{s_{i+1}},\dots,\bar{\pi}_{s_{i+j}}]\right]
	 & =  [X,[Y,Z]]= [Y,Z,X]\leq [Z,X,Y][X,Y,Z].
\end{align*}
As $(i+1,j-1)<(i,j)$ we may induct to find 
\begin{align*}
	[Z,X,Y] & =
  [[\bar{\pi}_{s_{i+j}},\bar{\pi}_{s_1},\dots,\bar{\pi}_{s_i}],
    [\bar{\pi}_{s_{i+1}},\dots,\bar{\pi}_{s_{i+j-1}}]]
 \leq \prod_{\sigma \in S_{i+j}}
	 [\bar{\pi}_{s_{1\sigma}},\dots, \bar{\pi}_{s_{(i+j)\sigma}}]. 
\end{align*}
Since $(i,j-1)<(i,j)$ we appeal once more to induction to show
\begin{align*}
	[X,Y,Z] & =
    \left[\left[[\bar{\pi}_{s_1},\dots,\bar{\pi}_{s_i}],
      [\bar{\pi}_{s_{i+1}},\dots,\bar{\pi}_{s_{i+j-1}}]\right],\bar{\pi}_{s_{i+j}}\right]\\
&  \leq \left[\prod_{\sigma \in S_{i+j-1}}
	[\bar{\pi}_{s_{1\sigma}},\dots, \bar{\pi}_{s_{(i+j-1)\sigma}}], \bar{\pi}_{i+j}\right]
 \leq \prod_{\sigma \in S_{i+j}}
	 [\bar{\pi}_{s_{1\sigma}},\dots, \bar{\pi}_{s_{(i+j)\sigma}}].
\end{align*}
Combining these three inclusions we observe that the formula holds for
$(i,j)$ and so by induction \eqref{eq:collect} holds.

Fix $s,t\in M$.  For path $(s_1,\dots,s_i)$ from $0$ to $s$ and $(s_{i+1},\dots,s_{i+j})$ from
$0$ to $t$, it follows that $s+t=s_1+\cdots+s_{i+j}$ and so
$(s_1,\dots,s_{i+j})$ is a path from $0$ to $s+t$.  So now it follows from \eqref{eq:collect} that:
\begin{align*}
	[\bar{\pi}_s,\bar{\pi}_t] & = \prod_{
		u\in \mathcal{G}_0^s
		v\in \mathcal{G}_0^t}
		[ [\pi_u],[\pi_v]]
	 \leq \prod_{w\in \mathcal{G}_0^{s+t}} [\pi_{w}]=\bar{\pi}_{s+t}.
\end{align*}
\end{proof}

The formula \eqref{eq:bar} is not sufficient to generate a filter as we must also guarantee that for every
$s\in M$, and $(s_1,\dots,s_d)\in \mathcal{G}_0^s$, $\bar{\pi}_s\leq \bar{\pi}_{s_1}\cap\cdots \cap \bar{\pi}_{s_d}$.  
For that we have needed further assumptions on $\pi$ and on $M$.

We will rely on commutative (pre-)ordered monoids $\langle M,+,\prec\rangle$ with minimal element $0$ and
satisfying:
\begin{equation}\label{def:star}
\textnormal{If $s\prec u$ and  $u=\sum_{i=1}^d u_i$,
then there exists $s_i\prec u_i$ with $s=\sum_{i=1}^d s_i$.}
\end{equation}
To show that a commutative monoid satisfies \eqref{def:star} it suffices to prove it for $d=2$.  

\begin{ex}
\begin{enumerate}[(i)]
\item Every cyclic monoid satisfies \eqref{def:star}.
\item For a family of commutative monoids satisfying \eqref{def:star}, the direct product also
satisfies \eqref{def:star}.  In particular direct products of cyclic monoids satisfy \eqref{def:star}.
\item $\mathbb{N}^d$ with the lexicographic well-ordering satisfies \eqref{def:star}.
\end{enumerate}
\end{ex}
\begin{proof}
For (i), note first that every cyclic monoid is isomorphic $C_{k,m}=\{c^i: 0\leq i<k+m\}$ for $k\in \mathbb{N}\cup\{\infty\}$ 
and $m\in \mathbb{Z}^+$, and where the product is 
$c^ic^j=c^{i+j}$ if $i+j<k$; else, $c^i c^j=c^{k+r}$ where $i+j-k=qm+r$, $0\leq r<m$. 
(Note $C_{\infty,m}\cong \mathbb{N}$.)  
Suppose that $s=c^i\prec u=c^j$ and $u=u_1 u_2$ with
$u_i=c^{e_i}$, for $0\leq i,j,e_1,e_2<k+m$.  If $s\prec u_1$ (i.e.: $i\leq e_1$), let $s_1=s$ and $s_2=1$. 
Otherwise, $u_1\prec s$ so $i\geq e_1$ and we let $s_1=u_1$, $s_2=c^{i-e_1}$ thus
$s_1 s_2=c^{i}=s$, $s_1\prec u_1$, $s_2\prec u_2$.

Now we prove (ii).  Let $\mathcal{F}$ be a family of commutative monoids satisfying \eqref{def:star}.
Let $s=(s^F: F\in \mathcal{F}),u=(u^F: F\in\mathcal{F})\in \prod\mathcal{F}$ with $s\prec u$.  It follows 
that for every $F\in\mathcal{F}$, $s^F\prec u^F$.
So if $u=u_1+u_2$ then $u^F=u_1^F+u_2^F$ so there are $s_1^F, s_2^F\in F$ such that
$s_i^F\prec u_i^F$ and $s=(s_1^F+s_2^F: F\in\mathcal{F})$.  

Finally we prove (iii).
Write $s=\sum_{i=1}^d s_i e_i$, $t=\sum_{i=1}^d t_i e_i$ and $u_j=\sum_{i=1}^d u_{ij} e_i$ with $s_i,t_i,u_{ij}\in \mathbb{N}$.  
The $i$-th row of the matrix
$U=[u_{ij}]$ sums to $t_i$. Assume $s\neq t$ so $s<t$ and there is a $c$ where $s_1=t_1,\dots, s_c=t_c$ and 
$s_{c+1}<t_{c+1}$.  For $1\leq i\leq c$, set $v_{ij}=u_{ij}$.
Next, since $s_{c+1}<t_{c+1}$, there are $v_{(c+1)j}\leq u_{(c+1)j}$ such that
$s_{c+1}=\sum_j v_{(c+1)j}$ and at least one $j_0$ exists such that $v_{(c+1)j_0}<u_{(c+1)j_0}$.  Finally, for each $c+1<i\leq d$, if
$s_i\leq t_i$ then choose $s_i=\sum_{j} v_{ij}$ with $v_{ij}\leq u_{ij}$ for each $i$;
otherwise, $s_i>t_i$.  So, set $v_{ij}=u_{ij}$ for all $j\neq j_0$ and $v_{ij_0}=u_{ij_0}+(s_i-t_i)$.

We claim the matrix $V=[v_{ij}]$ has the following properties: (i) for each $i$,
 $s_i=\sum_j v_{ij}$,
and (ii) for each $j$, $v_j=\sum_i v_{ij} e_i\leq u_j$ in the lexicographic order.
For (i), the first $c+1$ rows are elected in this manner as are any subsequent rows
where $s_i\leq t_i$.  If in a row $i$, $s_i>t_i$ then 
$\sum_{j} v_{ij}=u_{ij_0}+(s_i-t_i)+\sum_{j\neq j_0} u_{ij}=s_i-t_i+t_i=s_i$.
For (ii), the first $c$ coefficients of $v_j$ and $u_j$ agree so the first place
$v_j$ can differ from $u_j$ is if $v_{(c+1)j}<u_{(c+1)j}$.  If this occurs then
$v_j\leq u_j$.  Otherwise, for all $i$, $v_{ij}\leq u_{ij}$ and so $v_j\leq u_j$. 
\end{proof}

\begin{thm}\label{thm:generate}
Fix a commutative (pre-)ordered monoid $\langle M, +, \prec \rangle$ with minimal element $0$ and 
satisfying \eqref{def:star}. Fix a set ${\tt X}$ of generators for $M$ such
that for every $x\in {\tt X}$ and $y\in M$ if $y\prec x$ then $y\in {\tt X}$.
If $\pi:{\tt X}\to 2^G$ maps into the normal subgroups of $G$ and for every $s,u\in {\tt X}$, 
if $s\prec u$ then $\pi_{u}\leq \pi_s$, then $\bar{\pi}$ is a (pre-)ordered filter.
\end{thm}
\begin{proof}
Fix $s,u\in M$ with $s\prec u$ for the (pre-)-order on $M$.
By \eqref{def:star} for every $w=(u_1,\dots,u_d)\in \mathcal{G}_0^u$,  as $s\prec u$, 
there exists a decomposition $s=\sum_{i=1}^d s_i$ with $s_i\prec u_i$.  
By our assumption on ${\tt X}$, $x=(s_1,\dots,s_d)\in \mathcal{G}_0^s$.  By our assumptions
on $\pi$, $\pi_{u_i}\leq \pi_{s_i}$ and $[\pi_{w}]\leq [\pi_{x}]$.  Thus:
\begin{align*}
	\bar{\pi}_{u} & = \prod_{w\in \mathcal{G}_0^{s+t}} [\pi_{w}]
		\leq  \prod_{x\in\mathcal{G}_0^s} 
			[\pi_{x}] = \bar{\pi}_s.
\end{align*}
So $\pi$ is order-reversing.  Now for all $s,t\in M$, $s\prec s$ and $0\prec t$ so $s\prec s+t$.  Thus,
$\pi_{s+t}\leq \pi_s$.  Likewise, $\pi_{s+t}\leq \pi_t$.
Together with \lemref{lem:lower}
we see $\bar{\pi}$ is a (pre-)ordered filter.
\end{proof}

\section{Adjoint, centroid, and derivation filter refinements}
Here we introduce new filters by considering refinements of known filters $\phi:M\to 2^G$.
 Fix $s,t\in M$ and assume $M$ has property \eqref{def:star}.
The graded product of the Lie algebra $L=L(\phi)$ (see \eqref{eq:Lie-filter}) restricts to
a bimap (biadditive map) $[,]=[,]_{st}:L_{s}\times L_{t}\to L_{s+t}$.
We introduce three nonassociative rings  $\Adj([,])$, $\Der([,])$, and $\Cent([,])$ that
are new sources of $G$-invariant subgroups and capture various properties of
commutation in $G$.  The rings are:   
\begin{align*}
  \Adj([,]) & = \{ (X,Y)\in \End(L_s)\times \End(L_t)^{{\rm op}}:\\
  			 & \qquad \forall u\in L_s, \forall v\in L_t,\;[uX,v]= [u,Yv]\},\\
	\Cent([,] ) & = \{ (X,Y;Z) \in \End(L_s)\times \End(L_t)\times \End(L_{s+t}) :\\
		& \qquad \forall u\in L_s,\forall v\in L_t,\; [uX,v]+[u,vY] = [u,v]Z\},\quad \&\\
	\Der([,] ) & = \{ (X,Y;Z) \in \gl(L_s)\times \gl(L_t)\times \gl(L_{s+t}) :\\
		& \qquad \forall u\in L_s,\forall v\in L_t\; [uX,v]+[u,vY] = [u,v]Z\}.
\end{align*}
The {\em adjoint} ring $\Adj([,])$ is unital and associative. The {\em centroid} ring
 $\Cent([,])$ is unital, associative, and essentially commutative.\footnote{
Technically, $[,]_{st}$ factors through $L_s/L_t^{\bot}\times L_t/L_s^{\bot}\to [L_s,L_t]$, 
where $L_t^{\perp}=\{u\in L_s: [u,L_t]=0\}$ etc..  The centroid on that induced bimap is commutative 
\cite{Wilson:direct-I}*{Lemma 6.8(iii)}.} The {\em derivation} ring
$\Der([,])$ is a Lie ring.  The motivation to consider these rings is discussed in Section~\ref{sec:algebras}.

Now let $J$ be the Jacobson radical of $\Adj([,])$. Set $J^0=\Adj([,])$ and 
$J^{i+1}=J^i J$ for $i\geq 0$.  For each $i\in \mathbb{N}$, define 
$H_i$ so that $\phi_{s}^+\leq H_i\leq \phi_s$ and 
\begin{align}\label{def:alpha}
  H_i/\phi_{s}^+ & = L_s J^i.
\end{align}
Next, for $(u,i)\in {\tt X}=M\times \{0\}\cup \{u: u\prec s\}\times \mathbb{N}$, define
\begin{align*}
	 \alpha_u^i & = \left\{\begin{array}{cc}
		\phi_u & i=0 \textnormal{ or } u\neq s,\\
		H_i & u=s.
	\end{array}\right.
\end{align*}
\begin{lemma}\label{lem:hypo}
$(M\times\mathbb{N}, {\tt X}, \alpha)$ satisfies the hypotheses of \thmref{thm:generate}.
\end{lemma}
\begin{proof}
As $M$ and $\mathbb{N}$ satisfy \eqref{def:star}, so does $M\times \mathbb{N}$.
Evidently ${\tt X}$ generates $M \times \mathbb{N}$ as a monoid.  For $(x,i)\in {\tt X}$ 
and $(y,j)\in M$, if $(y,j)\prec (x,i)$ then $y\prec x$ and $j\leq i$.  In particular, either $i=j=0$
so that $(y,j)\in {\tt X}$ or $i>0$, $y\prec x\prec s$, and so $(y,j)\in {\tt X}$.  
Also $\alpha:{\tt X}\to 2^G$ maps into the normal subgroups of $G$.
Finally, if $(y,j)\prec (x,i)$ in ${\tt X}$ then $y\prec x$ and $j\leq i$.  If $x=y=s$ then
$\alpha_x^i=H_i\leq H_j=\alpha_y^j$.  Suppose $x\neq s$.  If $y\neq s$ 
then $\alpha_y^j=\phi_y\geq \phi_x=\alpha_x^i$.   If $y=s$ then $\alpha_y^j=H_j\geq \phi_s^+\geq \phi_x=\alpha_x^i$.
\end{proof}

In light of \thmref{thm:generate} and \lemref{lem:hypo},
we may now speak of the filter $\bar{\alpha}:M\times \mathbb{N}\to 2^G$ generated by $\alpha$.
We call this an  {\em adjoint} refinement.

Our construction of adjoint refinements of filters depends  
on the choice of filter and the selected homogeneous components of the Lie ring. 
A natural choice is to begin with a $\nu$ series $\nu:\mathbb{N}^d\to 2^G$.
Since $L_0$ represents a fixed set of operators it is not informative to refine $L_0$
so instead we refine the first $L_s\neq 0$ where $L_s\neq L_0$.
Of course $\nu$ series are well-ordered and so a small modification is necessary.
Assume $M=\mathbb{N}^d$ with the right-to-left lexicographic order (i.e. $e_{d}=(\dots,0,1)>e_{d-1}=(\dots,1,0)>\cdots$).
Define $\alpha$ as above only now restricted to the generating set ${\tt X}=\{u\in \mathbb{N}^{d+1} : u< (s+e_{d},0)\}$. 
(Note in $M$, $s^+=s+e_{d}$ which is why ${\tt X}$ is the appropriate choice.)
The resulting well-ordered filter $\bar{\alpha}$ is again a $\nu$-series, but now on $\mathbb{N}^{d+1}$.
We call this a {\em lex-least} adjoint refinement of $\nu$.

In general we can begin with the lower central series $\gamma$ on $\mathbb{N}$, or other standard $\mathbb{N}$-filters.
By recursive lex-least adjoint refinements we construct  longer and longer characteristic series indefinitely or until
the series stabilizes.  In that case we speak of the {\em stable lex-least} refinement.

Note that the analogous 
constructions based on the radicals of $\Cent([,])$ and $\Der([,])$ 
(equivalently the radical of the associative enveloping algebra of $\Der([,])$) 
determine functions $\sigma:{\tt X}\to 2^G$ and $\delta:{\tt X}\to 2^G$ and corresponding filters and $\nu$ series.
 
%
%
\subsection{Concrete examples}\label{sec:concrete}
We introduce some families of examples of lex-least refinements.  Usually to spot proper adjoint, centroid, 
or derivation refinements requires we setup and solve specific systems of 
linear equations and discern the structure of the appropriate rings.  There
are algorithms for that task \citelist{\cite{Wilson:find-cent}*{Sections 4  \& 5}\cite{BW:isom}\cite{deGraaf}\cite{GMT}}.
That is not always necessary and as the examples below show.

\begin{ex}
The $\gamma$ series of a finitely generated nonabelian free group has no proper lex-least
adjoint, derivation, or centroid refinement.
\end{ex}
\begin{proof}
In a free group of rank $r$, $\gamma_1/\gamma_2\cong \mathbb{Z}^r$, $\gamma_2/\gamma_3\cong \mathbb{Z}^{\binom{r}{2}}$ and $[,]:\gamma_1/\gamma_2\times \gamma_1/\gamma_2\to \gamma_2/\gamma_3$ is the exterior square $\mathbb{Z}^r\times \mathbb{Z}^r\to \mathbb{Z}^r\wedge \mathbb{Z}^r$.  The adjoint ring of $[,]$ is isomorphic to $\mathbb{Z}$ if $r\neq 2$ and $M_2(\mathbb{Z})$ if $r=2$; cf. \cite{Wilson:unique-cent}*{Sections 7.1 \& 7.6}.  In these rings the Jacobson radical is trivial.  Likewise, the centroid (which embeds in the adjoint ring) is  $\mathbb{Z}$.  Finally, the derivation algebra of the exterior square is $\mathfrak{gl}_r(\mathbb{Z})$ so its enveloping algebra
has a trivial radical.
\end{proof}

\begin{ex}\label{ex:longerseries}
Every finite $p$-group with $\eta_1/\eta_2$ of odd dimension and $\eta_2/\eta_3$ of dimension $2$ has a proper lex-least
adjoint refinement of the $\eta$ series.
\end{ex}

The \exref{ex:longerseries} explains the proper refinement later in \eqref{eq:4by4} and it 
generalizes in several ways.
We assert the following fact in the proof  of \exref{ex:longerseries}.

\begin{lemma}\label{lem:adj-ex}
Let $Z\in M_m(\mathbb{Z}_p)$ be the matrix with $Z_{i(m-i+1)}=1$ and $0$'s elsewhere.  Define a bimap $\circ:(\mathbb{Z}_p^m\oplus \mathbb{Z}_p^{m+1})\times (\mathbb{Z}_p^m\oplus \mathbb{Z}_p^{m+1})\to \mathbb{Z}_p^2$ by
\begin{align*}
	(u,v)\circ (x,y) & =(uFy^t-vF^t x, uGy^t-vG^t x^t)
\end{align*}
where $F=[Z,0], G=[0,Z]\in M_{m\times (m+1)}(\mathbb{Z}_p)$.  It follows that 
\begin{align*}
	\Adj(\circ) & = \left\{\left(\begin{bmatrix} aI_m & T \\ 0 & bI_{m+1}\end{bmatrix}, 
	\begin{bmatrix} b I_m & -T\\ 0 & a I_{m+1} \end{bmatrix}\right) : T \textnormal{ a Toeplitz matrix}\right\}.
\end{align*}
In particular, $J(\Adj(M))\cong \mathbb{Z}_p^{2m}>0$.
\end{lemma}

\begin{proof}[Proof of \exref{ex:longerseries}]
We argue along the lines of Bond \cite{Bond}*{pp. 608--611} and Vi{\v{s}}nevecki{\u\i} \cite{Vish}.
Let $V=\eta_1/\eta_2\cong \mathbb{Z}_p^{2m+1}$, and $W=\eta_2\cong\mathbb{Z}_p^{2}$.  Commutation produces an alternating $\mathbb{Z}_p$-bimap $[,]:\mathbb{Z}_p^{2m+1}\times \mathbb{Z}_p^{2m+1}\to \mathbb{Z}_p^{2}$,  equivalently, a pair of alternating forms on an odd-dimensional vector space.  

Using a classic result of Kronecker (generalized by others, compare \cite{Scharlau}) we know pairs of alternating forms are perpendicularly decomposable into indecomposable pairs of subspace $V=E_1\oplus \cdots \oplus E_s$.   As our vector space has odd dimension, at least one $E_i$ has odd dimension.  Furthermore, there is one type (up to equivalence) of odd-dimensional indecomposable pair of alternating forms. Specifically this is the bimap described in \lemref{lem:adj-ex}.
Notice $\Adj([,])$ restricted to $E_i$ is $e_i \Adj([,])e_i$, for $e_i$ the projection idempotent of $V$ onto $E_i$ with kernel 
$\bigoplus_{i\neq j} E_j$.  By \lemref{lem:adj-ex}, $e_i \Adj([,]) e_i\cong \Adj(\circ)$ has a nontrivial radical.  Thus, $\Adj([,])$
has a nontrivial radical and the lex-least adjoint refinement of $\eta$ is proper.\footnote{Each $E_j$ of even dimension has
$e_j\Adj([,])e_j\cong M_2(\mathbb{Z}_p[x]/(a_j(x)^{c_j}))$ for $a_j(x)$ an irreducible polynomial.  So the length
of an adjoint refinement is at least $\underset{j}{\max}\{1,c_j\}$.}
\end{proof}

\begin{ex}
Fix a finite nonabelian $p$-group $G$ that is not $2$-generated.
If $H\leq G$ where $[G:H]=p$ and $[H,H]\leq \eta_3(G)$, then $G$ has 
proper lex-least adjoint and derivation refinement of its $\eta$ series.
\end{ex}
\begin{proof}
Let $V=\eta_1/\eta_2$ and $W=\eta_2/\eta_3$.  So $[H/\eta_2,H/\eta_2]\equiv 0\in W$ and $H/\eta_2$ 
is a hyperplane in $V$ which is totally isotropic with respect to the commutation bimap $[,]:V\times V\to W$.  
Factoring out the radical of $[,]$, we arrive at the bimap described in \cite{Wilson:unique-cent}*{Lemma 7.13}.
That lemma shows  $J(\Adj([,]))>0$ when $\dim V>2$.  The argument for derivations is analogous.
 Hence the  adjoint and derivation refinements are proper.
\end{proof}

Our final family of examples is parametrized by commutative unital rings.
\begin{ex}\label{ex:Hei}
Let $R$ be an associative commutative unital ring with Jacobson radical $J$.  Consider the Heisenberg group 
\begin{equation}
	H(R) = \begin{bmatrix} 1 & R & R\\ . & 1 & R\\ . & . & 1 \end{bmatrix}.
\end{equation}
The lex-least adjoint refinement $\alpha:\mathbb{N}^2\to 2^{H(R)}$ of the $\gamma$ series has $\alpha_c^i(H)=1$ for $c>2$ and for $i\in\mathbb{N}$, 
\begin{align}\label{eq:ex-alpha-1}
	\alpha_{1}^i(H) & = \begin{bmatrix} 1 & J^i & R\\ . & 1 & J^i\\ . & . & 1 \end{bmatrix}, & 
	\alpha_2^i(H) & = \begin{bmatrix} 1 & . & J^i \\ . & 1 & . \\ . & . & 1 \end{bmatrix}.
\end{align}
In particular, if $J^c>J^{c+1}=0$, then the length of $\alpha$ is $2c+2$, whereas the $\gamma$ series has length $2$.\footnote{The length of a $\nu$ series counts the
number of non-redundant terms $L_s\neq 0$, for $s\neq 0$.}
\end{ex}
\begin{proof}
To see this is correct we notice that commutation on the first allowable 
graded component of 
$L(\gamma)$ amounts to $R^{\oplus 2}\times R^{\oplus 2}\to R$ where:
\begin{align*}
	(\forall & (a,b),(c,d)\in R^{\oplus 2}) & [(a,b),(c,d)] & = ad-bc
\end{align*}
Hence, $\Adj([,])$ is the ring
\begin{align*}
  \left\{\left(\begin{bmatrix} a & b\\ c & d\end{bmatrix},
  \begin{bmatrix} d & -b\\ -c & a \end{bmatrix}\right): a,b,c,d\in R\right\}
  \cong M_2(R).
\end{align*}
For $i\in\mathbb{N}$, $(R^{\oplus 2}) J^i(\Adj([,]))=(R^{\oplus 2})M_2(J^i(R))=(J^i)^{\oplus 2}$.  This gives us the $\alpha_1^i(H)$ 
described in \eqref{eq:ex-alpha-1} and the rest follows from \thmref{thm:generate}:
\begin{align*}
	\alpha_2^i & = \prod_{i=i_1+i_2} [\alpha_1^{i_1},\alpha_1^{i_2}]
		 = \left\langle \begin{bmatrix} 1 & . & J^{i_1} J^{i_2} \\ . & 1 & . \\ . & . & 1 \end{bmatrix} : i=i_1+i_2\right\rangle
		 =\begin{bmatrix} 1 & . & J^i \\ . & 1 & . \\ . & . & 1 \end{bmatrix}.
\end{align*} 
\end{proof}

 It might be speculated that
the lex-least adjoint refinement of upper unitriangular $(d\times d)$-matrix groups 
proceeds along similar lines to \exref{ex:Hei} and so it will be uninteresting for
matrices over fields.  To the contrary, the lex-least adjoint refinement of
the $\gamma$ series of the upper unitriangular $(4\times 4)$-matrices over a field 
is the following proper refinement (for proof see \exref{ex:longerseries}):
\begin{align}\label{eq:4by4}
\begin{bmatrix} 1 & * & * & *\\ . & 1 & * & * \\ . & . & 1 & * \\ . & . & . & 1
\end{bmatrix}
& >
\begin{bmatrix} 1 & * & * & *\\ . & 1 & . & * \\ . & . & 1 & * \\ . & . & . & 1
\end{bmatrix}
>
\begin{bmatrix} 1 & . & * & *\\ . & 1 & . & * \\ . & . & 1 & . \\ . & . & . & 1
\end{bmatrix}
 >
\begin{bmatrix} 1 & . & . & *\\ . & 1 & . & . \\ . & . & 1 & . \\ . & . & . & 1
\end{bmatrix}
 > 1.
\end{align}

\begin{remark}\label{rem:unipotent}
J. Maglione of
Colorado State University has computed the stable lex-least adjoint refinement
of upper unitriangular $(d\times d)$-matrices.  The
result is that the $\gamma$ series (which has length $d-1$) is refined to 
characteristic $\nu$ series  of length $d^2/4+\Theta(d)$ with each factor in the series
of dimension at most $2$ over the field.
Further work and a Magma \cite{Magma} implementation is underway.
\end{remark}

\subsection{Positive logarithmic proportions of proper refinements}
We now give some attention to the proportion of finite $p$-groups that exhibit a proper lex-least refinement.
The number $f(p^n)$ of pairwise nonisomorphic groups of size $p^n$ is known for small values of $n$, but asymptotically we
only know logarithmic estimates, e.g.: $\frac{2}{27}n^3+Cn^2\leq \log_p f(p^n)\leq \frac{2}{27}n^3+C'n^{2.5}$ for constants $C$ 
and $C'$  \cite{BNV:enum}*{Chapter 1}.  
Similarly granularity occurs when counting pairwise nonisomorphic rings
\cite{Neretin}.   This means that proportions of finite groups are not generally quantifiable except on a logarithmic scale.
This is the context of our main result in this section.

\begin{thm}\label{thm:count}
For each prime $p$ and $n\gg 0$, there are at least $p^{2n^3/729+\Omega(n^2)}$ pairwise nonisomorphic
groups of order $p^n$ that have $\eta$ series of length $2$ and a lex-least
adjoint refinement of length at least $6$.
\end{thm}

In other words,  \thmref{thm:count} says that as $n\to \infty$, a positive logarithmic proportion of 
finite groups of size $p^n$ have proper adjoint refinements.  
Specifically: 
$\frac{2n^3/729+C''n^2}{2n^3/27+C'n^{2.5}}\to \frac{1}{27}$ in our result but the quantity could be larger.
Of course this count is logarithmic and so the actual proportion may well tend to zero.

\begin{lemma}\label{lem:ring-count}
There are at least $p^{2n^3/27+\Omega(n^2)}$ pairwise nonisomorphic local commutative associative unital rings $R$ of order $p^n$ with $J(R)>J^2(R)>0$.
\end{lemma}

\begin{proof}
Let $\circ:V\times V\to W$ be a nontrivial symmetric bimap of elementary abelian $p$-groups $V$ and $W$.  Set 
$R=R(\circ)=\mathbb{Z}_p\oplus V\oplus W$ as an additive group and equip $R$ with the following distributive product:
\begin{align*}
	(s,v,w)\cdot (s',v',w') & = (ss', sv'+vs', sw'+v\circ v'+ws').
\end{align*}
This product is  associative, commutative, and $(1,0,0)$ is the identity.\footnote{Writing $(s,v,w)$ as 
a formal matrix $\left[\begin{smallmatrix} s & v & w\\ . & s & v \\ . & . & s\end{smallmatrix}\right]$, the operations mimic matrix
operations.}  Furthermore, $J=0\oplus V\oplus W$ is a nilpotent ideal so
$J$ is contained in the Jacobson radical $J(R)$ of $R$.  Furthermore, $R/J\cong \mathbb{Z}_p$ is semisimple proving $J=J(R)$.
Finally, $J^2=0\oplus 0\oplus (V\circ V)>0$ as $\circ$ is nontrivial.  

To every pair $\circ,\diamond$ of bimaps $V\times V \to W$, with 
$W=V\circ V=V\diamond V$, an isomorphism $f:R(\circ)\to R(\diamond)$
induces isomorphisms $(f_V:V\to V, f_W:W\to W)$ by identifying $V=J/J^2$
and $W=J^2$.  Furthermore, for all 
$v,v'\in V$, $vf_V\circ v'f_V=(v\circ v')f_W$.  Therefore, the number $g(V,W)$ of pairwise nonisomorphic 
rings of the form $R(\circ)$ has the following lower bound.  Let $S^2V=V\otimes V/\langle u\otimes v-v\otimes u: u,v\in V\rangle$,
\begin{align*}
	g(V,W) & \geq \frac{|\hom(S^2V, W)|}{|\GL(V)\times \GL(W)|} 
		 \geq p^{\frac{1}{2}(\dim V)^2\dim W -(\dim V)^2-(\dim W)^2}.
\end{align*}
Let $d=\dim V$ so that $\dim W=n-1-d$.  As $n\to \infty$, the maximum of $p^{1/2\cdot d^2(n-1-d)-d^2-(n-1-d)^2}$ occurs for $d$ near $2n/3$ which produces
the lower bound of $p^{2n^3/27+\Omega(n^2)}$. (This method of counting is owed to Higman \cite{BNV:enum}*{Chapter 2}.)
\end{proof}

\begin{proof}[Proof of \thmref{thm:count}]
Consider Heisenberg groups.

If $H(R)\cong H(S)$ for local commutative rings $R$ and $S$ then $M_2(R)\cong \Adj([,]_R)\cong \Adj([,]_S)\cong M_2(S)$.
Fix an isomorphism $f:M_2(R)\to M_2(S)$.  Since $R$ is local $E=\left[\begin{smallmatrix} 1 & 0 \\ 0 & 0 \end{smallmatrix}\right]$
is a primitive idempotent of $M_2(R)$ and therefore so is $Ef=F$.  Thus $f$ induces an isomorphism $R\cong EM_2(R)E\to FM_2(S)F\cong S$.  Consequently, for every isomorphism type of commutative associative 
local ring $R$ of size $p^{m}$ 
we obtain a distinct isomorphism type of group $H(R)$ of size $p^{3m}$. 

By \lemref{lem:ring-count} there are at least $p^{\frac{2}{27}m^3+\Omega(m^2)}$ commutative associative local rings $R$ of size $p^m$, for $m\gg 0$.  
So for $n\gg 0$ and $n=3m+r$, $0\leq r<3$, there are 
$p^{2n^3/(27)^2+\Omega(n^2)}$ pairwise nonisomorphic groups $G$ of order
$p^{3m}$ that have $\eta$ series of length $2$ but an adjoint refinement of length at least $6$ (\exref{ex:Hei}).  For any such group $G$, the length of the 
$\eta$ series for $G\times \mathbb{Z}_p^r$ is unchanged and the adjoint
refinement can only increase in length.  Thus, the claim holds for all $n\gg 0$.
\end{proof}

\section{Implications to group-isomorphism testing}\label{sec:iso}

The ``nilpotent quotient'' algorithm is arguably the leading method for group-isomorphism testing 
and for finding generators of automorphism groups of $p$-groups.  The idea is evident in lectures of G. Higman 
\cite{Higman:chic}*{pp. 10--12}, but was rediscovered in greater generality   
by M. F. Newman \cite{Newman} and refined extensively by E.A. O'Brien \cite{OBrien:first} and others. 
We summarize the framework to present our own adaptations.

A {\em characteristic} function $G\mapsto \tau(G)$ on groups $G$ sends $G$ to a subgroup $\tau(G)$ such that
every isomorphism $f:G\to H$ satisfies $\tau(G)f=\tau(H)$.
The first step of a nilpotent quotient algorithm is to specify for each $n\in \mathbb{N}$, characteristic functions $\tau_n$
where $\tau_n(G)/\tau_{n+1}(G)$ is elementary abelian.
Usually this is the series $\eta$ of $G$ but our intention is to use lex-least refinements of $\eta$. 
To detect isomorphisms between finite $p$-groups $G$ and $H$
the algorithm recursively builds isomorphisms $f_n:G/\tau_n(G)\to H/\tau_n(H)$, or proves 
that at some stage there is no isomorphism.  

The base case is isomorphism testing of elementary abelian groups which is straight-forward.  
Within the recursion, the isomorphism $f_n$ is used to create {\em covering groups} $C_n$ that satisfy the following.
\begin{enumerate}[(i)]
\item $G/\tau_n(G)\cong C_n/\tau_n(C_n)\cong H/\tau_n(H)$.
\item There are $M,N \normaleq \tau_n(C_n)$ with $C_n/M\cong G/\tau_{n+1}(G)$ and $C_n/N\cong H/\tau_{n+1}(H)$. 
\item For all isomorphisms $f:G/\tau_{n+1}(G)\to H/\tau_{n+1}(H)$, $\exists g\in \Aut(C_n)$, $Mg=N$ inducing $f$.
\end{enumerate}
The algorithm's work is to find $g\in \Aut(C_{n})$ where $N=Mg$. Such a $g$
induces an isomorphism $f_{n+1}:G/\tau_{n+1}(G)\to H/\tau_{n+1}(H)$ that allows the algorithm to increase $n$.
If this fails, by (iii) we know $G\not\cong H$.
(For computing automorphisms the task is instead to find generators of the stabilizer in $\Aut(C_{n})$ of $M$.)
There are many heuristics to find $g\in \Aut(C_n)$ with $Mg=N$, but most settings lead to
exhaustive searching in sets as large as $p^{O(d_n^2)}$, $d_n=[\tau_{n-1}(G),\tau_n(G)]$.  So the running time
is at worst $|G|^{O(\log |G|)}$.

To improve performance of this exhaustive search it is enough to have a series $\tau$ with small $\Aut(C_n)$-orbits on 
the subgroups of $\tau_n(C_{n})/\tau_{n+1}(C_n)$.
We propose $\nu$ series, such as the adjoint, centroid, or derivation 
refinements of $\eta$.
These series can decrease the size of the $\Aut(C_n)$-orbits in two ways.  
First the size of sections of longer series are correspondingly smaller.
Secondly, the automorphism groups must respect 
ring-theoretic properties of adjoints, centroids, and derivations.   Indeed, in some extreme cases it has been 
shown that the $\Aut(C_n)$-orbits are smaller than $p^{d_n}$.  That lead to isomorphism
tests for large families of $p$-groups that run in $O(\log^6 |G|)$ steps (compared to the 
$|G|^{O(\log |G|)}$ steps without considering adjoints) \cite{LW:iso}*{Theorem 2}. 

We show how nilpotent quotient  algorithms adapt to use the $\nu$ series we have described. 
Let $F[x_1,\dots,x_d]$ denote a free group on $d$ generators.
\begin{thm}
Let $G$ and $H$ be finite $p$-groups with $\nu$ a fixed lex-least adjoint, centroid,
or derivation refinement of their $\eta$ series.  Suppose that for some $(i,s)\in \mathbb{N}^2$, $i>1$, 
$G/\nu_i^s(G)\cong F[x_1,\dots,x_d]/R_i\cong H/\nu_i^s(H)$. Let $\nu_j^{t}(G)^{-1}$ be the preimage 
of $\nu_1^t(G)$ in $F=F[x_1,\dots,x_d]$. Set
\begin{align*}
	 U_i=[R_i,F]R_i^p \prod_{s+1=t_1+\cdots+t_i} [(\nu_1^{t_1}(G))^{-1},\dots,(\nu_1^{t_i}(G))^{-1}].
\end{align*}
It follows that $C_i=F/U_i$ is a covering group for the pair $(G/\nu^{s+1}_i(G),H/\nu^{s+1}_i(H))$.
\end{thm}
\begin{proof}
First, $G/\eta_{i+1}(G)$ is a quotient of $F/[R_i,F]R_i^p$ \cite{OBrien:first}*{Theorem 2.2}.  Since 
$\eta_{i+1}(G)=\nu_{i+1}^0(G)\leq \nu^{s+1}_i(G)\leq \nu_i^s(G)$, it follows that
$G/\nu_i^{s+1}(G)$ is a quotient  $C_i/M$ with $M\leq R_i$.  Likewise $H/\nu_i^{s+1}(H)$ 
is a quotient $C_i/N$ with $N\leq R_i$.  This resolves properties (i) and (ii) of a covering group.

Let $g:C_i\to G/\nu_i^{s+1}(G)$ be an epimorphism with kernel $M$ and $h:C_i\to H/\nu_i^{s+1}(H)$ 
an epimorphism with kernel $N$.  Suppose there is an isomorphism $f:G/\nu_{i}^{s+1}(G)\to H/\nu_i^{s+1}(H)$. 
Thus $f$ induces a function $j\mapsto w_j\in F$ such that $(U_i x_j)g f=(U_i w_j)h$.  As $F$ is free, 
we obtain a homomorphism $f':F\to F$ defined by $x_j\mapsto w_j$.  Now $f$ is an isomorphism and 
$G\mapsto \nu_1^t(G)$ is characteristic, so it follows that $U_if'= U_i$ and so $f'$ factors through the 
homomorphism $\hat{f}:C_i\to C_i$ given by $(U_i x_j)\hat{f}=U_i w_j$.  Finally, $\hat{f}$ is an 
isomorphism sending $M$ to $N$, as $C_i=\langle U_i x_1,\dots,U_i x_d\rangle=\langle U_i w_1,\dots, U_i w_d\rangle$.  

For an expanded argument of the same nature consider \cite{OBrien:first}*{Theorem 2.5}.
\end{proof}

\section{Closing remarks}

\subsection{Why use adjoints, centroids, and derivations?}\label{sec:algebras}
We have resisted discussing the reasons to consider adjoints, centroids, and derivations.  
We close with a few hints of their importance.

The adjoint ring in our generality was introduced in \cite{Wilson:unique-cent} 
to explain the nature of central products of $p$-groups.  Notice
 $[]=[,]_{st}:L_s\times L_t\to L_{s+t}$ factors through $L_s \otimes_{\Adj([,])} L_t$ uniquely.  Thus, ring-theoretic 
properties of $\Adj([,])$, e.g. radicals, idempotents, and nilpotent elements, have strong implications on the
commutation in $G$.  At times this is sufficient to distinguish groups up to isomorphism  \cite{LW:iso}.  
A detailed treatment of adjoints in general is found in \cite{Wilson:division}.  

One useful feature of the derivation algebra is that we can exponentiate nilpotent elements of $\Der([,])$ 
to produce pseudo-isometries of $[,]$.  If a Lazard-Mal'cev type correspondence is available
then these also lift to automorphisms of $G$.  Unfortunately one still deals with the issues of exponentiation in positive characteristic.  

The fundamental property of the centroid is that $[,]$ is bilinear with respect to $\Cent([,])$. For example, 
the centroid of the group $H(R)$ is $R$. Centroids in this form arose to describe direct product decompositions of $p$-groups, cf. \cite{Wilson:direct-I}*{Section 6.4}.  The associated filter for the centroid ring  informs us of direct factors.  

\subsection{The role of characteristic abelian subgroups}
There is a theme connecting many examples of proper adjoint refinements which is 
poorly understood.
If a finite nilpotent group $G$ has a proper lex-least adjoint refinement of its $\gamma$ series then there exists a proper characteristic
subgroup $\gamma_2<H< \gamma_1$ such that $[H,H]\leq \gamma_3$.  This is because the adjoint ring 
$A=\Adj([,]_1:L_1\times L_1\to L_2)$ ($L_i=\gamma_i/\gamma_{i+1})$) has a proper Jacobson radical $J$.
Hence for some power $J^i>0$, $J^{2i}=0$.  Thus, $[L_1J^i,L_1J^i]=[L_1 J^{2i}, L_1]=[0,L_1]=0$.  So take $H=\alpha_1^i$.
The examples in Section~\ref{sec:concrete} where discovered from this observation.  It is not known what additional properties are needed on $H$ to guarantee
that $H$ is part of an adjoint refinement.  Ultimately, since computing adjoints is linear algebra it is more likely in practice that
we will discover $H$ by computing $A$ rather than have foreknowledge of $H$ to imply properties of $A$.

\subsection{Many more examples}
The lex-least refinements of adjoint, centroid, and derivation type we have considered here are only of 
$\nu$ series $\nu:\mathbb{N}^d\to 2^G$.  Also they only depended on the bimap of 
the leading nontrivial homogeneous component.  That bimap is unchanged by central extensions $N.G$ in which 
$N$ is contained in the Frattini subgroup.  Thus \thmref{thm:count} should imply the
existence of many more examples.

We also emphasize again that our refinements can be repeated indefinitely and perhaps
in different ways.  Since the increase in length can be substantial
(consider \remref{rem:unipotent}) we are tempted to consider a notion of ``class'' that counts the longest possible series
by some canonical refinement process.

\subsection{Why stop at filters?}
In our introduction we remarked that characteristic series are desired so that we can constrain the properties of $\Aut(G)$.
In our constructions we have computed various nonassociative rings $\Adj([,])$, $\Der([,])$, and $\Cent([,])$ on which $\Aut(G)$ is
represented.  The action of $\Aut(G)$ on these nonassociative rings is much more informative than
the longer series that they induce.  However in practice the automorphisms of nonassociative rings are extremely difficult
groups to produce whereas the refinements we describe here are reasonable to compute but remain informative.  
A related work of the author and Brooksbank attempts to constrain automorphism groups further
and shows the difficulties \cite{BW:autotopism}.

\section*{Acknowledgments}

I am grateful to T. Doresey, J. Maglione, and C. R. B. Wright for many helpful remarks and discussions,
and to the referee for candid advice.

\end{document}